\documentclass[11pt]{amsart}
\usepackage[usenames, dvipsnames, table]{xcolor}
\usepackage{amssymb, tikz, tikz-cd, pgfplots, tabto, bm, bbm, url, mathdots, adjustbox} 
\usepackage{verbatim}
\usepackage{hyperref}
\usepackage{enumitem}

\usepackage{geometry}
\theoremstyle{definition}
\newtheorem{theorem}{Theorem}[section]
\newtheorem{example}[theorem]{Example}

\newtheorem{definition}[theorem]{Definition}

\newtheorem{proposition}[theorem]{Proposition}

\newtheorem*{ack}{Acknowledgment}

\numberwithin{equation}{section}

\allowdisplaybreaks

\newcommand{\la}{\scalebox{1.9}[1]{$\langle$}}
\newcommand{\laa}{\la\hspace{-1.9mm}\la}
\newcommand{\ra}{\scalebox{1.9}[1]{$\rangle$}}
\newcommand{\raa}{\ra\hspace{-1.9mm}\ra}

\newcommand{\del}{\partial}
\newcommand{\delbar}{\overline{\partial}}

\newcommand{\bu}{\bullet}
\newcommand{\mc}{\mathcal }
\newcommand{\mC}{{\mathcal C}}

\newcommand{\mcS}{\mathcal{S}}
\newcommand{\uo}{^{(1)}} 
\newcommand{\ut}{^{(2)}} 
\newcommand{\al}{\alpha}
\newcommand{\be}{\beta}
\newcommand{\ga}{\gamma}
\newcommand{\de}{\delta}
\newcommand{\ka}{\kappa}
\newcommand{\om}{\omega}
\newcommand{\wt}{\widetilde}

\newcommand{\HA}{H_{\mathcal A}}
\newcommand{\HBC}{H_{\mathcal B\mathcal C}}
\newcommand{\xb}{\overline{x}}
\newcommand{\yb}{\overline{y}}
\newcommand{\zb}{\overline{z}}
\newcommand{\wb}{\overline{w}}
\newcommand{\ct}{\widetilde{c}}

\title{Cyclic ABC Massey products}

\author[T.~Tradler]{Thomas~Tradler}
  \address{Thomas Tradler,
  Department of Mathematics, New York City College of Technology, City University of New York, 300 Jay Street, Brooklyn, NY 11201}
  \email{ttradler@citytech.cuny.edu}

\author[S.O.~Wilson]{Scott O.~Wilson}
  \address{Scott O. Wilson,
  Department of Mathematics, Queens College, City University of New York, 65-30 Kissena Blvd, Flushing, NY 11367}
  \email{scott.wilson@qc.cuny.edu}


\begin{document}
\maketitle
\begin{abstract} 
This paper refines the notion of cyclic Massey products to the bi-graded setting, just as quadruple ABC Massey products refine the notion of quadruple Massey products. The result, we call ``cyclic ABC Massey products,'' are in general non-trivial and contain information different from the quadruple ABC Massey products.
\end{abstract}

\setcounter{tocdepth}{1} 

\section{Introduction}
In this note we define certain higher order cyclic Massey products in the bi-graded setting. Massey products are defined for differential graded associative algebras as higher order cohomology  operations, which are related to detecting higher knotting of links \cite{Mas58,Mas68}.

When an algebra is a \emph{commutative} differential graded algebra (CDGA), one can apply cyclicity considerations to define further variants of these Massey products that make use of commutativity. In fact, for any non-negative integers $k,\ell$ the authors defined a higher cyclic Massey product $\la.\ra_{k,\ell}$ in \cite[section 5.1]{PTW}; the case of $k=\ell=1$ is recalled in section \ref{SEC:cyclic-CDGA} below. The cyclic Massey products are not independent since $\la.\ra_{k,\ell}$ can essentially be expressed in terms of $\la.\ra_{\ell,k}$ via a cyclicity relation, and, in the case $\ell=0$ they reduce to the usual $(k+2)$th Massey products.

A commutative bi-differential bi-graded algebra (CBBA) has a more appropriate notion of Massey product, due to Angella and Tomassini \cite{AT},  which uses the extra bi-differential structure to obtain a more refined version of the triple Massey product. These are called ABC Massey products, as they take three Bott-Chern cohomology classes as inputs and gives a subset of Aeppli cohomology as an output. In \cite{SM}, it is shown that these ideas extend to quadruple ABC Massey product, and in principle extend to even higher versions. It's natural then to expect that the family of cyclic Massey products for CDGAs also have a refined notion of cyclic ABC Massey products $\laa.\raa_{k,\ell}$ for CBBAs. This was the question posed by the second author at the Critical Math Workshop, after which the first author supplied the details in this paper, making it explicit in the case $k =\ell = 1$. We'll denote $\laa.\raa_{1,1}$ simply by $\laa.\raa$.

In section \ref{SEC:cyclic-CBBA} we show that for any four Schweitzer cohomology classes $[\al_i]\in H^{1}(\mcS_{p_i,q_i}), i=1,2,3,4$, the cyclic ABC Massey product is a subset $\laa[\al_1],[\al_2],[\al_3],[\al_4]\raa\subseteq H^{-1}(\mcS_{p,q})$, which is natural with respect to morphisms of CBBAs, and which satisfies the expected cyclicity condition \eqref{EQU:cyclicity}. These cyclic ABC Massey products map to the usual cyclic Massey products via the following inclusion; see proposition \ref{PROP:ABC-regular}
\begin{equation*}
 \frac{(-1)^{|\al|+|\ga|}}{4} (\del-\delbar) \Big(\laa[\al],[\be],[\ga],[\de]\raa\Big)\subseteq \la [\al],[\be],[\ga],[\de])\ra.
\end{equation*}
We give a non-trivial example of the cyclic ABC Massey product using the holomorphic analogue of the filiform nilmanifold in example \ref{EXA:holomorphic-filiform}.

This paper did not use AI in any form, including its conception, proofs, or writing. It might be interesting to use AI to construct nontrivial examples of compact manifolds with non-trivial cyclic ABC-Massey products.  

\begin{ack}
We thank Alexander Milivojevi\'c for useful conversations concerning the topic of this paper. The first author  was partially supported by a PSC-CUNY Award. The second author acknowledges support provided by the Simons Foundation Award program for Mathematicians and a PSC-CUNY Award, jointly funded by
The Professional Staff Congress and The City University of New York. 
\end{ack}

\section{Cyclic Massey product for a CDGA}\label{SEC:cyclic-CDGA}

We recall the basics of cyclic Massey products for a CDGA. More precisely, we spell out the specifics of the $k=1$ and $\ell=1$ case as defined in \cite[Section 5.1]{PTW}.

Let $(A^\bu,\cdot, d:A^\bu\to A^{\bu+1})$ be a CDGA, and denote by $H^\bu(A)=\frac{Ker \,d}{Im \, d}$ its cohomology. Let $\al, \be, \ga,\de\in A$ be closed elements in $A$, $d\al=d\be=d\ga=d\de=0$. Then, a cyclic defining system $S$ for $\al, \be, \ga,\de$ is given by elements $\kappa, \lambda, \mu,\nu, \eta,\varrho\in A$ satisfying
\begin{equation}\label{EQU:cyclic-defining-system}
\left\{
\begin{matrix}
d\kappa=(-1)^{|\al|}\cdot \al\cdot \be \\
d\lambda=(-1)^{|\be|}\cdot \be\cdot \ga\\
d\mu=(-1)^{|\ga|}\cdot \ga\cdot \de\\
d\nu=(-1)^{|\de|}\cdot \de\cdot \al
\end{matrix}
\hspace{1cm}
\begin{matrix}
d\eta=(-1)^{|\kappa|}\cdot \kappa\cdot \ga+(-1)^{|\al|}\cdot \al\cdot \lambda \\
d\varrho=(-1)^{|\mu|}\cdot \mu\cdot \al+(-1)^{|\ga|}\cdot \ga\cdot \nu
\end{matrix}
\hspace{1cm}
\begin{tabular}{ccccc}  \cline{3-5}
& \multicolumn{1}{c|}{} & \multicolumn{1}{c|}{$\varrho$}  & \multicolumn{1}{c|}{$\nu$}  & $\al$ \\   \cline{3-4}
& \multicolumn{1}{c|}{} & \multicolumn{1}{c|}{$\mu$}  & $ \de$ \\ \cline{1-3}
\multicolumn{1}{|c|}{$\eta$} & \multicolumn{1}{c|}{$\lambda$} & $\ga$ &  \\ \cline{1-2}
\multicolumn{1}{|c|}{$\kappa$} & $\be$ & &   \\ \cline{1-1} 
\multicolumn{1}{|c}{$\al$} &  &  &
\end{tabular}
\right.
\end{equation}
Here, we require that the above are equations of homogeneous degrees, i.e., if $\al$ is of degree $|\al|$, etc., then the degree of $\kappa$ is $|\kappa|=|\al|+|\be|-1$, etc., and the degree of $\eta$ is $|\eta|=|\kappa|+|\ga|-1=|\al|+|\lambda|-1$, etc.

With this, \cite[Proposition 5.3]{PTW} states that there is an induced map 
\begin{equation}
\la.\ra:H^{r_1}(A)\times H^{r_2}(A)\times H^{r_3}(A)\times H^{r_4}(A)\to \mc P(H^{r_1+r_2+r_3+r_4}(A))
\end{equation}
taking four cohomology classes and mapping them to a subset of cohomology, given by
\begin{align}
&\la[\al],[\be],[\ga],[\de]\ra :=\{[\mc C_S]:
S \text{ is a cyclic defining system for } \al, \be, \ga,\de\},\quad\text{where}\\
&\quad \mc C_S:=(-1)^{|\eta|+1} \eta\cdot \de +(-1)^{|\be|+\epsilon_0} \be \cdot\varrho+
(-1)^{|\kappa|+1} \kappa\cdot \mu+(-1)^{|\lambda|+\epsilon_0}\lambda\cdot \nu
\end{align}
and $(-1)^{\epsilon_0}
=(-1)^{|{\al}|\cdot (|{\be}|+|{\ga}|+|{\de}|)+|{\al}|+|{\be}|+|{\ga}|+|{\de}|}$.

By \cite[Proposition 5.4]{PTW}, if $f:A\to B$ is a map of CDGAs, then we have an inclusion
\begin{equation}\label{EQU:CDGA-naturality}
f_*\big(\la [\al],[\be],[\ga],[\de] \ra\big) \subseteq \la [f(\al)],[f(\be)],[f(\ga)],[f(\de)]\ra\text{ in }H^\bu(B).
\end{equation}

\begin{proposition}\label{PROP:qi=>cyclicMIP_equality}
If $f:A\to B$ is a quasi-isomorphism, then the inclusion \eqref{EQU:CDGA-naturality} is an equality.
\end{proposition}
\begin{proof}
Let $\wt\kappa, \wt\lambda, \wt\mu,\wt\nu, \wt\eta,\wt\varrho\in B$ be a cyclic defining system for $\wt\al=f(\al), \wt\be=f(\be), \wt\ga=f(\ga),\wt\de=f(\de)$ denoted by $\wt S$. Since $f(\al\cdot \be)=f(\al)\cdot f(\be)=(-1)^{|\al|}d(\wt\kappa)$ and $f$ is a quasi-isomorphism, $\al\cdot \be$ is exact, and so there exists a $\kappa\in A$ with $d(\kappa)=(-1)^{|\al|}\al\cdot \be$. Since $d(\wt\kappa-f(\kappa))=0$, there exists a closed $\ka^c\in A$ such that $[f(\ka^c)]= [\wt\ka-f(\ka)]$, i.e., their representatives differ by some exact term, $f(\ka+\ka^c)-\wt\ka=d(\wt\ka^\circ)$ for some $\wt\ka^\circ\in B$. Setting $\ka_1=\ka+\ka^c$ and ${\wt\ka}_1=\wt\ka+ d(\wt\ka^\circ)$ gives $d(\ka_1)=(-1)^{|\al|}\al\cdot \be$, $d({\wt\ka}_1)=(-1)^{|\al|}f(\al)\cdot f(\be)$, and $f(\ka_1)=\wt\ka_1$. Moreover, setting $\wt\eta_1=\wt\eta+(-1)^{|\wt\ka|}\wt \ka^\circ\cdot \wt\ga$ makes $\wt\kappa_1, \wt\lambda, \wt\mu,\wt\nu, \wt\eta_1,\wt\varrho$ into a new cyclic defining system $\wt S_1$ for $f(\al), f(\be), f(\ga),f(\de)$ with $[\mc C_{\wt S_1}]=[\mc C_{\wt S}]$, since the difference $\mc C_{\wt S_1}-\mc C_{\wt S}=(-1)^{|\wt\eta|+1+|\wt\ka|}\cdot \wt\ka^\circ\cdot \wt\ga\cdot \wt\de+(-1)^{|\wt\kappa|+1}\cdot d(\wt\ka^\circ)\cdot \wt\mu=d((-1)^{|\wt\kappa|+1}\cdot \wt\ka^\circ\cdot \wt\mu)$ is exact, where we used that $(-1)^{|\wt\eta|+1+|\wt\kappa|}=(-1)^{|\wt\ga|}$.

Repeating the above for $\wt\lambda$, $\wt\mu$ and $\wt\nu$, we end up with elements $\wt\kappa_2, \wt\lambda_2, \wt\mu_2,\wt\nu_2, \wt\eta_2,\wt\varrho_2\in B$, which form a new cyclic defining system $\wt S_2$ for $f(\al), f(\be), f(\ga),f(\de)$ with $[\mc C_{\wt S}]=[\mc C_{\wt S_2}]$, as well as elements $\kappa_2, \lambda_2, \mu_2,\nu_2 \in A$ which satisfy the left four equations in \eqref{EQU:cyclic-defining-system} and  $f(\kappa_2)=\wt\kappa_2,f(\lambda_2)= \wt\lambda_2, f(\mu_2)=\wt\mu_2,f(\nu_2)=\wt\nu_2$. Now, just as above, $f((-1)^{|\kappa|}\cdot \kappa_2\cdot \ga+(-1)^{|\al|}\cdot \al\cdot \lambda_2)=d(\wt\eta_2)$ shows that there exists an $\eta\in A$ with $d\eta=(-1)^{|\kappa|}\cdot \kappa_2\cdot \ga+(-1)^{|\al|}\cdot \al\cdot \lambda_2$. Then, $d(\wt\eta_2-f(\eta))=0$, thus there exists a closed $\eta^c\in A$ and there exists $\wt\eta^\circ\in B$ with $f(\eta+\eta^c)-\wt\eta_2=d(\wt\eta^\circ)$. Setting $\eta_3=\eta+\eta^c$ satisfies $d\eta_3=d\eta=(-1)^{|\kappa|}\cdot \kappa_2\cdot \ga+(-1)^{|\al|}\cdot \al\cdot \lambda_2$ and setting $\wt\eta_3= \wt\eta_2+d(\wt\eta^\circ)$ satsfies $f(\eta_3)=\wt\eta_3$ and gives a cyclic defining system $\wt S_3$ given by $\wt\kappa_2, \wt\lambda_2, \wt\mu_2,\wt\nu_2, \wt\eta_3,\wt\varrho_2$ with $[\mc C_{\wt S}]=[\mc C_{\wt S_3}]$ since the difference $\mc C_{\wt S_3}-\mc C_{\wt S_2}=(-1)^{|\wt \eta|+1}\cdot d(\wt\eta^\circ)\cdot \delta=d((-1)^{|\wt \eta|+1}\cdot \wt\eta^\circ\cdot \delta)$ is exact.

Repeating the above for $\wt\varrho_2$, we obtain a cyclic defining systems $S_4$ and $\wt S_4$ for $\al, \be,\ga,\de$ and $f(\al), f(\be),f(\ga),f(\de)$, repscetively, where the terms of the cyclic defining system map to each other under $f$, and, $[\mc C_{\wt S_4}]=[\mc C_{\wt S}]$. Thus, $f([\mc C_{S_4}])=[f(\mc C_{S_4})]=[\mc C_{\wt S_4}]=[ \mc C_{\wt S}]$ lies in the lefthand side of \eqref{EQU:CDGA-naturality}, showing the opposite inclusion ``$\supseteq$'' in \eqref{EQU:CDGA-naturality}.
\end{proof}

\section{Cyclic ABC Massey product for a CBBA}\label{SEC:cyclic-CBBA}

We now give an elementary account of cyclic ABC Massey products for a CBBA. Let $(A^{\bu,\bu},\cdot, \del:A^{\bu,\bu}\to A^{\bu+1,\bu}, \delbar:A^{\bu,\bu}\to A^{\bu,\bu+1})$ be a CBBA. For a fixed bi-degree $(p,q)$, recall the Schweitzer complex $\mcS_{p,q}=\mcS^\bu_{p,q}(A)$ shown below 
is given as follows:
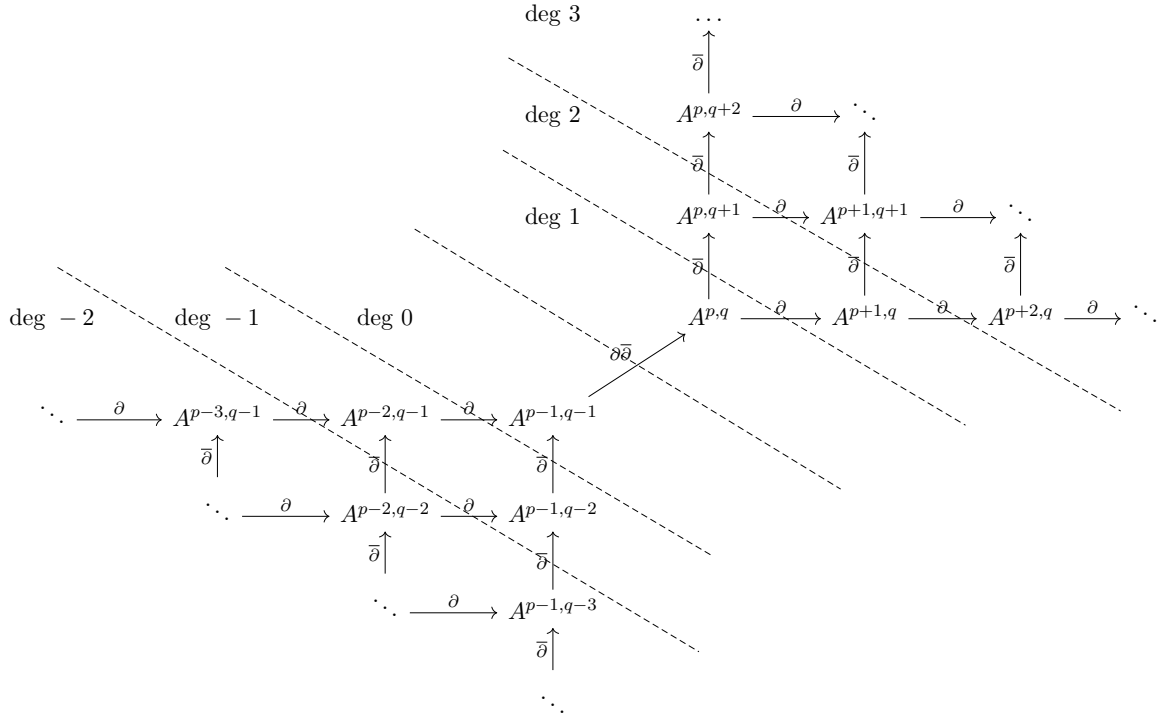
\[
\begin{adjustbox}{width=\textwidth+12mm}
\begin{tikzcd}
&&&\text{deg }3& \dots\\
&&&\text{deg }2& A^{p,q+2}\arrow[u,"\delbar"]\arrow[r,"\del"]
&  \ddots \\
&&&\text{deg }1& A^{p,q+1}\arrow[u,"\delbar"]\arrow[r,"\del"] 
\arrow[ru, dashed, no head, start anchor={[xshift=-3.8cm, yshift=2.2cm]}, end anchor={[xshift=4.5cm, yshift=-4.5cm]}]
&A^{p+1,q+1}\arrow[u,"\delbar"]\arrow[r,"\del"]&  \ddots \\
\text{deg }-2&\text{deg }-1&\text{deg }0&& A^{p,q}\arrow[u,"\delbar"]\arrow[r,"\del"] 
\arrow[ru, dashed, no head, start anchor={[xshift=-3.8cm, yshift=2.4cm]}, end anchor={[xshift=2cm, yshift=-3.1cm]}]
&A^{p+1,q}\arrow[u,"\delbar"]\arrow[r,"\del"]& A^{p+2,q}\arrow[u,"\delbar"]\arrow[r,"\del"] & \ddots & & \\
\ddots \arrow[r, "\del"]& A^{p-3,q-1}\arrow[r, "\del"] 
\arrow[rrdd, dashed, no head, start anchor={[xshift=-3cm, yshift=2.7cm]}, end anchor={[xshift=3cm, yshift=-1cm]}]
& A^{p-2,q-1} \arrow[r, "\del"] 
\arrow[rd, dashed, no head, start anchor={[xshift=-3cm, yshift=2.7cm]}, end anchor={[xshift=3.2cm, yshift=-1cm]}]
& A^{p-1,q-1}\arrow[ru, "\del\delbar"]  
\arrow[ru, dashed, no head, start anchor={[xshift=-2.8cm, yshift=2.7cm]}, end anchor={[xshift=2.5cm, yshift=-2.5cm]}]
& & \\
& \ddots\arrow[u,"\delbar"]\arrow[r, "\del"] & A^{p-2,q-2}\arrow[u,"\delbar"] \arrow[r, "\del"] & A^{p-1,q-2} \arrow[u,"\delbar"] & & \\
& & \ddots \arrow[u,"\delbar"]\arrow[r, "\del"] & A^{p-1,q-3}\arrow[u,"\delbar"] & &  \\
&&& \ddots\arrow[u,"\delbar"]& & 
\end{tikzcd}
\end{adjustbox}
\]
\[
\mcS_{p,q}^{k}(A)=\left\{
\begin{matrix}
A^{p,q+k-1}\oplus A^{p+1,q+k-2}\oplus A^{p+2,q+k-3}\oplus\dots\oplus A^{p+k-1,q} & \text{, for }k\geq 1 \\
A^{p+k-1,q-1}\oplus A^{p+k,q-2}\oplus A^{p+k+1,q-3}\oplus\dots\oplus A^{p-1, q+k-1} & \text{, for }k\leq 0 
\end{matrix}\right.
\]
Here, the differential $\mcS_{p,q}^0\to \mcS_{p,q}^1$ is given by $\del\delbar$, while for $k\neq 0$, the differential $\mcS_{p,q}^k\to \mcS_{p,q}^{k+1}$ is given by $\del+\delbar$, which in the case $k<0$ is furthermore restricted to the appropriate spaces. Below, we state the Schweitzer cohomology $H^k(\mcS_{p,q})$ in degrees $k=1,0,-1$.
\begin{align*}
H^1(\mcS_{p,q})&=\frac{Ker(\del|_{A^{p,q}}) \cap Ker(\delbar|_{A^{p,q}})}{Im (\del\delbar|_{A^{p-1,q-1}})}
\\
H^0(\mcS_{p,q})&=\frac{Ker (\del \delbar|_{A^{p-1,q-1}})}{Im (\del|_{A^{p-2,q-1}}) + Im (\delbar|_{A^{p-1,q-2}})}
\\
H^{-1}(\mcS_{p,q})&=\frac{\{(x^{(1)},x^{(2)})\in A^{p-2,q-1}\oplus A^{p-1,q-2}\,:\, \del(x^{(1)})+\delbar(x^{(2)})=0\}}{Im (\del|_{A^{p-3,q-1}}) + Im ((\del+\delbar)|_{A^{p-2,q-2}}) + Im (\delbar|_{A^{p-1,q-3}})}
\end{align*}
 Note, that $H^1(\mcS_{p,q})=\HBC^{p,q}(A)$ is the Bott-Chern cohomology, and $H^0(\mcS_{p,q}) =\HA^{p-1,q-1}(A)$ is the Aeppli cohomology, respectively given in appropriate degrees. For better readability, we will sometimes suppress the indices $p$ and $q$ when these are implicitly given.

\begin{definition}\label{DEF:defining-system-C_S}
Let $\al, \be, \ga,\de\in A$ be elements in $A$ which are both $\del$-closed and $\delbar$-closed, $\del\al=\delbar\al=\del\be=\delbar\be =\del\ga=\delbar\ga=\del\de=\delbar\de=0$. Then, a cyclic ABC defining system for $\al, \be, \ga,\de$ is given by elements $\kappa, \lambda, \mu,\nu, \eta,\eta',\varrho,\varrho'\in A$ satisfying
\begin{equation}\label{EQU:cyclic-ABC-defining-system}
\left\{
\begin{matrix}
\del\delbar\kappa= \al\cdot \be \\
\del\delbar\lambda= \be\cdot \ga\\
\del\delbar\mu= \ga\cdot \de\\
\del\delbar\nu= \de\cdot \al
\end{matrix}
\hspace{1cm}
\begin{matrix}
\del\eta+\delbar\eta'= \kappa\cdot \ga-\al\cdot \lambda\\
\del\varrho+\delbar\varrho'= \mu\cdot \al-\ga\cdot \nu
\end{matrix}
\hspace{1cm}
\begin{tabular}{ccccc}  \cline{3-5}
& \multicolumn{1}{c|}{} & \multicolumn{1}{c|}{$\varrho,\varrho'$}  & \multicolumn{1}{c|}{$\,\,\,\nu\,\,\,$}  & $\,\,\,\al\,\,\,$ \\   \cline{3-4}
& \multicolumn{1}{c|}{} & \multicolumn{1}{c|}{$\mu$}  & $ \de$ \\ \cline{1-3}
\multicolumn{1}{|c|}{$\eta, \eta'$} & \multicolumn{1}{c|}{$\,\,\,\lambda\,\,\,$} & $\ga$ &  \\ \cline{1-2}
\multicolumn{1}{|c|}{$\kappa$} & $\be$ & &   \\ \cline{1-1} 
\multicolumn{1}{|c}{$\al$} &  &  &
\end{tabular}
\right.
\end{equation}
We denote the total degree of an element $\om\in A^{p_\om,q_\om}$ by $|\om|=p_\om+q_\om$. Then, we require that the above equations are homogeneous in their bi-degrees, i.e., if $\al\in A^{p_\al,q_\al}$, etc., then $\kappa\in A^{p_\al+p_\be-1,q_\al+q_\be-1}$, etc., and $\eta\in A^{p_\ka+p_\ga-1,q_\ka+q_\ga}=A^{p_\al+p_\be+p_\ga-2, q_\al+q_\be+q_\ga-1}$, etc.

If we denote the cyclic ABC defining system given by elements $\kappa, \lambda, \mu,\nu, \eta,\eta',\varrho,\varrho'\in A$ by $S$, then we define the cyclic ABC staircase product $\mC_S$ to be
\begin{align}
\label{EQU:CS^3} \mC_S&:=(\mC\uo_S,\mC\ut_S)\quad \in A^{p-2,q-1}\oplus A^{p-1,q-2}\,\,\, ,\quad\quad\quad \text{where}\\
\label{EQU:CS^1} \mC\uo_S&:=\eta\cdot \de +(-1)^{\epsilon+|\be|+1}\cdot\be \cdot\varrho +(\delbar\kappa)\cdot \mu +(-1)^{\epsilon+1}\cdot(\delbar\lambda)\cdot \nu\\
\label{EQU:CS^2} \mC\ut_S&:= \eta'\cdot \de +(-1)^{\epsilon+|\be|+1}\cdot\be \cdot\varrho'  +(-1)^{|\kappa|}\cdot \kappa\cdot (\del\mu) +(-1)^{\epsilon+|\lambda|+1}\cdot\lambda\cdot (\del\nu)
\end{align}
with $p=p_\al+p_\be+p_\ga+p_\de$ and $q=q_\al+q_\be+q_\ga+q_\de$, and $(-1)^{\epsilon}=(-1)^{|{\al}|\cdot (|{\be}|+|{\ga}|+|{\de}|)}$.
\end{definition}

\begin{theorem}\label{THM:cyclic-ABC-Massey}
There is a well-defined map
\[
\laa.\raa:H^{1}(\mcS_{p_1,q_1})\times H^{1}(\mcS_{p_2,q_2})\times H^{1}(\mcS_{p_3,q_3})\times H^{1}(\mcS_{p_4,q_4})\to \mc P\Big(H^{-1}(\mcS_{p,q})\Big)
\]
where $p=p_1+p_2+p_3+p_4$ and $q=q_1+q_2+q_3+q_4$, taking four Bott-Chern cohomology classes in $H^{1}(\mcS_{p_j,q_j})=\HBC^{p_j,q_j}(A)$ for $j=1,2, 3, 4$, and mapping them to a subset of Schweitzer cohomology in degree $-1$, given, for any inputs $[\al],[\be],[\ga],[\de]$, by
\begin{multline}
\laa[\al],[\be],[\ga],[\de]\raa:=\{[\mC_S]\in H^{-1}(\mcS_{p,q}): 
\\ S \text{ is a cyclic ABC defining system for } \al, \be, \ga,\de\}
\quad \subseteq H^{-1}(\mcS_{p,q})
\end{multline}
We call $\laa[\al],[\be],[\ga],[\de]\raa$ the cyclic ABC Massey product of $[\al],[\be],[\ga],[\de]$.

These products satisfy the cyclicity condition
\begin{equation}\label{EQU:cyclicity}
\laa[\al],[\be],[\ga],[\de]\raa=-(-1)^{(|\al|+|\be|)\cdot(|\ga|+|\de|)}\cdot\laa[\ga],[\de],[\al],[\be]\raa
\end{equation}
\end{theorem}
\begin{proof}
Let $S$ be a cyclic ABC defining system for $\al, \be, \ga,\de$ given by $\kappa, \lambda, \mu,\nu, \eta,\eta',\varrho,\varrho'\in A$ which satisfy \eqref{EQU:cyclic-ABC-defining-system}. Then, $[\mC_S]$ is an element in $H^{-1}(\mcS_{p,q})$, since, using \eqref{EQU:cyclic-ABC-defining-system}, we have:
\[
\del\mC\uo_S+\delbar \mC\ut_S
=(\kappa \ga-\al \lambda)\delta-(-1)^{\epsilon}\beta( \mu \al-\ga \nu)
+(\del\delbar\kappa)\mu-\kappa(\del\delbar\mu)
-(-1)^\epsilon(\del\delbar\lambda)\nu+(-1)^\epsilon\lambda(\del\delbar\nu)
=0
\]

The cyclicity condition follows since cyclic rotation of the cyclic ABC defining system $S$ in \eqref{EQU:cyclic-ABC-defining-system} gives a cyclic ABC defining system $S^\circlearrowleft$, for $\ga,\de,\al,\be$ for which we can compute $\mC_S+(-1)^{(|\al|+|\be|)\cdot(|\ga|+|\de|)}\cdot \mC_{S^\circlearrowleft}=(\del+\delbar)(\kappa\cdot \mu+(-1)^{\epsilon+1}\lambda\cdot \nu)$, so that the classes $[\mC_S]=[-(-1)^{(|\al|+|\be|)\cdot(|\ga|+|\de|)}\cdot \mC_{S^\circlearrowleft}]$ in $H^{-1}(\mcS_{p,q})$.

It remains to show that $\laa.\raa$ is well-defined on the input classes in $H^1(\mcS_{p_j,q_j})$. If we add the exact term $\del\delbar \xi$ to $\beta$, then we consider the new cyclic ABC defining system $S^\xi$ given by
\[
\begin{tabular}{ccccc}  \cline{3-5}
& \multicolumn{1}{c|}{} & \multicolumn{1}{c|}{$\varrho,\varrho'$}  & \multicolumn{1}{c|}{$\,\,\,\nu\,\,\,$}  & $\,\,\,\al\,\,\,$ \\   \cline{3-4}
& \multicolumn{1}{c|}{} & \multicolumn{1}{c|}{$\mu$}  & $ \de$ \\ \cline{1-3}
\multicolumn{1}{|c|}{$\eta, \eta'$} & \multicolumn{1}{c|}{$\lambda+\xi\gamma$} & $\ga$ &  \\ \cline{1-2}
\multicolumn{1}{|c|}{$\kappa+\al\xi$} & \rule{0pt}{12pt}$\be+\del\delbar\xi$ & &   \\ \cline{1-1} 
\multicolumn{1}{|c}{$\al$} &  &  &
\end{tabular}
\]
Computing $\mC\uo_{S^\xi}-\mC\uo_S=(-1)^{\epsilon+|\be|+1}(\del\delbar\xi)\varrho+\delbar(\al\xi) \mu +(-1)^{\epsilon+1}\delbar(\xi\ga) \nu$ and $\mC\ut_{S^\xi}-\mC\ut_S=(-1)^{\epsilon+|\be|+1}(\del\delbar\xi) \varrho'  +(-1)^{|\kappa|} \al\xi (\del\mu) +(-1)^{\epsilon+|\lambda|+1}\xi\ga (\del\nu)$, shows that the difference $\mC_{S^\xi}-\mC_S$ equals the Schweitzer differential $ A^{p-3,q-1}\oplus A^{p-2,q-2}\oplus A^{p-1,q-3}\ni (a,b,c)\mapsto (\del a+\delbar b, \del b +\delbar c)$ applied to $(-1)^{\epsilon+|\be|+1}\cdot ((\delbar\xi)\varrho,(-1)^{|\be|+1}\xi(\delbar\varrho'),-(\del\xi)\varrho')$, showing that $[\mC_{S^\xi}]=[\mC_S]$ in $H^{-1}(\mcS_{p,q})$.

Similarly, when adding $\del\delbar\zeta$ to $\ga$, we consider the defining system $S^\zeta$ given by
\[
\begin{tabular}{ccccc}  \cline{3-5}
& \multicolumn{1}{c|}{} & \multicolumn{1}{c|}{\rule{0pt}{12pt} $\varrho-(\delbar\zeta)\nu,\varrho'-(-1)^{|\zeta|}\zeta (\del\nu)$}  & \multicolumn{1}{c|}{$\,\,\,\nu\,\,\,$}  & $\,\,\,\al\,\,\,$ \\   \cline{3-4}
& \multicolumn{1}{c|}{} & \multicolumn{1}{c|}{$\mu+\zeta\de$}  & $ \de$ \\ \cline{1-3}
\multicolumn{1}{|c|}{$\eta+(-1)^{|\kappa|}\kappa(\delbar \zeta), \eta'+(\del\kappa)\zeta$} & \multicolumn{1}{c|}{$\lambda+\be\zeta$} &\rule{0pt}{12pt} $\ga+\del\delbar\zeta$ &  \\ \cline{1-2}
\multicolumn{1}{|c|}{$\kappa$} & $\be$ & &   \\ \cline{1-1} 
\multicolumn{1}{|c}{$\al$} &  &  &
\end{tabular}
\]
Thus, we get $\mC\uo_{S^\zeta}-\mC\uo_S=(-1)^{|\kappa|}\kappa(\delbar\zeta)\de+(-1)^{\epsilon+|\be|}\be(\delbar\zeta)\nu+(\delbar \kappa)\zeta\de-(-1)^\epsilon\delbar(\be\zeta)\nu=\delbar(\kappa\zeta\de)$ and $\mC\ut_{S^\zeta}-\mC\ut_S=(\del\kappa)\zeta\de+(-1)^{\epsilon+|\be|+|\zeta|}\be\zeta(\del\nu)+(-1)^{|\kappa|}\kappa\del(\zeta\de)-(-1)^{\epsilon+|\lambda|}\be\zeta(\del\nu)=\del(\kappa\zeta\de)$, which shows that the difference $\mC_{S^\zeta}-\mC_S$ equals the Schweitzer differential applied to $(0,\kappa\zeta\de,0)$, which in turn shows that $[\mC_{S^\zeta}]=[\mC_S]$ in $H^{-1}(\mcS_{p,q})$.

Independence of adding exact terms to $\al$ or $\de$ now also follows via the cyclicity condition.
\end{proof}

\begin{proposition}
If $f:A\to B$ is a map of CBBAs, then for any Bott-Chern classes $[\al],[\be],[\ga],[\de]$, there is an inclusion of subsets of $H^{-1}(\mcS_{p,q}(B))$,
\begin{equation}\label{EQU:CBBA-naturality}
f_*\big(\laa [\al],[\be],[\ga],[\de] \raa\big) \subseteq \laa [f(\al)],[f(\be)],[f(\ga)],[f(\de)]\raa
\end{equation}

If, moreover, $f:A\to B$ induces isomorphisms on Schweitzer cohomologies $H^k(\mcS^\bu_{p,q}(A))\to H^k(\mcS^\bu_{p,q}(B))$ for $k=1,0,-1$ and for all $p$ and $q$, then the inclusion \eqref{EQU:CBBA-naturality} is an equality.
\end{proposition}
\begin{proof}
For a cyclic ABC defining system $S$ for $\al, \be, \ga, \de$ as in \eqref{EQU:cyclic-ABC-defining-system}, applying $f$ to each element in \eqref{EQU:cyclic-ABC-defining-system} gives a cyclic ABC defining system $S^f$ for $f(\al), f(\be), f(\ga), f(\de)$ with $f(\mC_{S})=\mC_{S^f}$, which shows the inclusion \eqref{EQU:CBBA-naturality}.

The proof for the second assertion is analogous to the one for proposition \ref{PROP:qi=>cyclicMIP_equality}. Let $\wt\kappa$, $\wt\lambda$, $\wt\mu$, $\wt\nu$, $\wt\eta$, $\wt\eta'$, $\wt\varrho$, $\wt\varrho' \in B$ be a cyclic ABC defining system for $\wt\al=f(\al), \wt\be=f(\be), \wt\ga=f(\ga),\wt\de=f(\de)$ denoted by $\wt S$. Since $f(\al\cdot \be)=f(\al)\cdot f(\be)=\del\delbar(\wt\kappa)$ and $f$ is an isomorphism on Bott-Chern cohomology $H^1(\mcS)$, $\al\cdot \be$ is $\del\delbar$-exact, and so there exists a $\kappa\in A$ with $\del\delbar(\kappa)=\al\cdot \be$. Since $\del\delbar(\wt\kappa-f(\kappa))=0$ and $f$ is an isomorphism on Aeppli cohomology $H^0(\mcS)$, there exists a $\del\delbar$-closed $\ka^c\in A$ such that the Aeppli classes $[f(\ka^c)]= [\wt\ka-f(\ka)]$ are equal, i.e., their representatives differ by a term in the denominator of Aeppli cohomology, $f(\ka+\ka^c)-\wt\ka=\del(\wt\ka^\circ)+\delbar(\wt\ka'^\circ)$ for some $\wt\ka^\circ,\wt\ka'^\circ\in B$. Setting $\ka_1=\ka+\ka^c$ and ${\wt\ka}_1=\wt\ka+ \del(\wt\ka^\circ)+\delbar(\wt\ka'^\circ)$ gives $\del\delbar(\ka_1)=\al\cdot \be$, $\del\delbar({\wt\ka}_1)=f(\al)\cdot f(\be)$, and $f(\ka_1)=\wt\ka_1$. Moreover, setting $\wt\eta_1=\wt\eta+\wt \ka^\circ\cdot \wt\ga$ and $\wt\eta'_1=\wt\eta'+\wt \ka'^\circ\cdot \wt\ga$ makes $\wt\kappa_1, \wt\lambda, \wt\mu,\wt\nu, \wt\eta_1, \wt\eta'_1,\wt\varrho, \wt\varrho'$ into a new cyclic ABC defining system $\wt S_1$ for $f(\al), f(\be), f(\ga),f(\de)$ with $[\mc C_{\wt S}]=[\mc C_{\wt S_1}]$ as classes in $H^{-1}(\mcS_{p,q})$, since $\mc C_{\wt S_1}-\mc C_{\wt S}=(\wt\ka^\circ\wt\ga\wt\de+\delbar(\del\wt\ka^\circ)\wt\mu,\wt\ka'^\circ\wt\ga\wt\de+(-1)^{|\wt\ka|}(\del\wt\ka^\circ+\delbar\wt\ka'^\circ)(\del\wt\mu))$ equals to the Schweitzer boundary applied to $((-1)^{|\wt\ka^\circ|}\wt\ka^\circ(\delbar\wt\mu),(\del\wt\ka^\circ)\wt\mu, (-1)^{|\wt\ka|}\wt\ka'^\circ(\del\wt\mu))$.

Repeating the above for $\wt\lambda$, $\wt\mu$ and $\wt\nu$, we end up with elements $\wt\kappa_2, \wt\lambda_2, \wt\mu_2,\wt\nu_2, \wt\eta_2,\wt\eta'_2,\wt\varrho_2, \wt\varrho'_2\in B$, which form a new cyclic ABC defining system $\wt S_2$ for $f(\al), f(\be), f(\ga),f(\de)$ with $[\mc C_{\wt S}]=[\mc C_{\wt S_2}]$, as well as elements $\kappa_2, \lambda_2, \mu_2,\nu_2 \in A$ which satisfy the left four equations in \eqref{EQU:cyclic-ABC-defining-system} and  $f(\kappa_2)=\wt\kappa_2,f(\lambda_2)= \wt\lambda_2, f(\mu_2)=\wt\mu_2,f(\nu_2)=\wt\nu_2$. Now, $f( \kappa_2\cdot \ga- \al \cdot\lambda_2)=\del\wt\eta_2+\delbar\wt\eta'_2$ shows that $\kappa_2 \cdot\ga- \al \cdot\lambda_2$ is zero in Aeppli cohomology, and, since $f$ is an isomorphism on Aeppli cohomology, there exists $\eta,\eta'\in A$ with $\del\eta+\delbar\eta'=\kappa_2\cdot \ga- \al\cdot \lambda_2$. Then, $\del(\wt\eta_2-f(\eta))+\delbar(\wt\eta'_2-f(\eta'))=0$, and, using the isomorphism $f$ on $H^{-1}(\mcS)$, there exist $\eta^c,\eta'^c \in A$ so that $(\eta^c,\eta'^c)$ is closed, in the sense that $\del\eta^c+\delbar\eta'^c=0$, with $[(f(\eta^c),f(\eta'^c))]=[(\wt\eta_2-f(\eta),\wt\eta'_2-f(\eta'))]$ in $H^{-1}(\mcS(B))$. On the cochain level this means that there are elements $\wt\eta^\circ,\wt\eta'^\circ,\wt\eta''^\circ\in B$ with $(\del\wt\eta^\circ+\delbar \wt\eta'^\circ, \del \wt\eta'^\circ+\delbar\wt\eta''^\circ) =(f(\eta+\eta^c)-\wt\eta_2,f(\eta'+\eta'^c)-\wt\eta'_2)$. Setting $\eta_3=\eta+\eta^c$ and $\eta'_3=\eta'+\eta'^c$ satisfies $\del\eta_3+\delbar\eta'_3=\del\eta+\delbar\eta'= \kappa_2 \cdot\ga- \al \cdot\lambda_2$, and setting $\wt\eta_3= \wt\eta_2+\del(\wt\eta^\circ)+\delbar(\wt\eta'^\circ)$ and $\wt\eta'_3= \wt\eta'_2+\del(\wt\eta'^\circ)+\delbar(\wt\eta''^\circ)$ gives $f(\eta_3)=\wt\eta_3$ and $f(\eta'_3)=\wt\eta'_3$. Moreover, we have a new cyclic ABC defining system $\wt S_3$ given by $\wt\kappa_2, \wt\lambda_2, \wt\mu_2,\wt\nu_2, \wt\eta_3,\wt\eta'_3, \wt\varrho_2, \wt\varrho'_2$ with $[\mc C_{\wt S}]=[\mc C_{\wt S_3}]$ since the difference $\mc C_{\wt S_3}-\mc C_{\wt S_2}=(\del(\wt\eta^\circ\cdot \de)+\delbar(\wt\eta'^\circ\cdot \de),\del(\wt\eta'^\circ\cdot \de) + \delbar (
\wt\eta''^\circ\cdot \de))$ is exact and thus vanishes in $H^{-1}(\mcS(B))$.

Repeating the above for $\wt\varrho_2$ and $\wt\varrho'_2$, we obtain cyclic ABC defining systems $S_4$ and $\wt S_4$ for $\al, \be,\ga,\de$ and $f(\al), f(\be),f(\ga),f(\de)$, respectively, where the terms of the cyclic ABC defining system map to each other under $f$, and, $[\mc C_{\wt S_4}]=[\mc C_{\wt S}]$. Thus, $f([\mc C_{S_4}])=[f(\mc C_{S_4})]=[\mc C_{\wt S_4}]=[ \mc C_{\wt S}]$ lies in the lefthand side of \eqref{EQU:CBBA-naturality}, showing the opposite inclusion ``$\supseteq$'' in \eqref{EQU:CBBA-naturality}.
\end{proof}

If $(A^{\bu,\bu},\cdot,\del, \delbar)$ is a CBBA, then the totalization $A^{r}=\bigoplus_{p+q=r} A^{p,q}$ with differential $d=\del+ \delbar$ gives a CDGA $(A^{\bu}, \cdot, d)$. For $\al\in A$, denote by $d^c(\al)=\frac 1 2 (\del-\delbar)(\al)$, so that we get $d\circ d^c=-\del\circ\delbar$. We have the following maps from the Schweitzer cohomology $H^k(\mcS_{p,q})$ in degrees $k=1$, $0$, and $-1$ to the cohomology $H^\bu(A)$ of $A$ with respect to $d$ in the specified degrees. 

\begin{definition}
Define $\iota:H^{1}(\mcS_{p,q})\to H^{p+q}(A)$ to be given by $\iota([\al])=[\al]$. Moreover, we also define the induced maps $d^c:H^{0}(\mcS_{p,q})\to H^{p+q-1}(A), d^c([\al])=[\frac 1 2 (\del-\delbar)(\al)]$, and $d^c:H^{-1}(\mcS_{p,q})\to H^{p+q-2}(A), d^c([(\al^{(1)},\al^{(2)})])=[\frac 1 2 (\del-\delbar)(\al^{(1)}+\al^{(2)})]$. 

A straightforward check shows that these three maps are well-defined linear maps.
\end{definition}

\begin{proposition}\label{PROP:ABC-regular}
There is an inclusion from the cyclic ABC Massey product $\laa.\raa$ to the cyclic Massey product $\la.\ra$. More precisely, we have:
\begin{equation}\label{EQU:DeRham-inclusion}
(-1)^{|\al|+|\ga|} \frac{d^c}{2}  \Big(\laa[\al],[\be],[\ga],[\de]\raa\Big)\subseteq
\la [\al],[\be],[\ga],[\de])\ra
\end{equation}
\end{proposition}
\begin{proof}
Let $\kappa, \lambda, \mu,\nu, \eta,\eta',\varrho,\varrho'\in A$ be a cyclic ABC defining system for $\al, \be, \ga,\de$ as in \eqref{EQU:cyclic-ABC-defining-system}. Then the following is a cyclic defining system for $\al, \be, \ga,\de$:
\[
\renewcommand{\arraystretch}{1.4} \begin{tabular}{ccccc}  \cline{3-5}
& \multicolumn{1}{c|}{} & \multicolumn{1}{c|}{$(-1)^{|\de|+1} \cdot \frac{d^c}{2}(\varrho+\varrho')$}  & \multicolumn{1}{c|}{$(-1)^{|\de|+1}\cdot d^c(\nu)$}  & $\,\,\,\al\,\,\,$ \\  \cline{3-4}
& \multicolumn{1}{c|}{} & \multicolumn{1}{c|}{$(-1)^{|\ga|+1}\cdot d^c(\mu)$}  & $ \de$ \\ \cline{1-3}
\multicolumn{1}{|c|}{$(-1)^{|\be|+1}\cdot \frac{d^c}{2}(\eta+\eta')$} & \multicolumn{1}{c|}{$(-1)^{|\be|+1}\cdot d^c(\lambda)$} & $\ga$ &  \\ \cline{1-2}
\multicolumn{1}{|c|}{$(-1)^{|\al|+1}\cdot d^c(\kappa)$} & $\be$ & &   \\ \cline{1-1} 
\multicolumn{1}{|c}{$\al$} &  &  &
\end{tabular} \renewcommand{\arraystretch}{1}
\]
since, for example, $d((-1)^{|\al|+1} d^c(\kappa))=(-1)^{|\al|}\del\delbar(\kappa)=(-1)^{|\al|}\al\be$, and, with $|\ka|=|\al|+|\be|-2$,
\begin{multline*}
d\Big((-1)^{|\be|+1}\cdot {\frac{d^c}{2}}(\eta+\eta')\Big)
=\frac{(-1)^{|\be|}} 2 \del\delbar(\eta+\eta')
=\frac{(-1)^{|\be|}} 2\Big(-\delbar(\kappa\gamma-\al\lambda-\delbar\eta')+\del(\kappa\gamma-\al\lambda-\del\eta)\Big)
\\
=(-1)^{|\kappa|+1}\cdot \Big((-1)^{|\al|+1}\frac{(\del-\delbar)}{2}(\kappa)\Big)\cdot \gamma+(-1)^{|\al|}\cdot \al\cdot \Big((-1)^{|\be|+1}\frac{(\del-\delbar)}{2}(\lambda))\Big).
\end{multline*}

Evaluating the terms appearing on the lefthand side of \eqref{EQU:DeRham-inclusion} using \eqref{EQU:CS^3} gives:
\begin{align*}
(-1)^{|\al|+|\ga|}  {\frac{d^c}{2}}  \Big(\eta\de+\eta'\de\Big)
=&(-1)^{|\eta|+1+1} \cdot  \Big((-1)^{|\be|+1} {\frac{d^c}{2}} (\eta+\eta')\Big)\cdot \de,
\\
(-1)^{|\al|+|\ga|} {\frac{d^c}{2}} \Big((-1)^{\epsilon+|\be|+1} \Big(\be\varrho+\be\varrho'\Big)\Big)
=&(-1)^{|\be|+\epsilon_0}\cdot \be\cdot \Big((-1)^{|\de|+1}  {\frac{d^c}{2}} (\varrho+\varrho')\Big),
\end{align*}
since $(-1)^{\epsilon_0}=(-1)^{\epsilon+|\al|+|\be|+|\ga|+|\de|}$, as well as:
\begin{multline*}
(-1)^{|\al|+|\ga|} {\frac{d^c}{2}} \Big((\delbar\kappa)\cdot \mu+(-1)^{|\kappa|}\cdot \kappa\cdot (\del\mu)\Big)\\
=(-1)^{|\kappa|}\cdot \Big((-1)^{|\al|+1}\cdot d^c(\kappa)\Big)\cdot \Big((-1)^{|\ga|+1}\cdot d^c(\mu)\Big)
+d\Big(\frac{(-1)^{|\al|+|\ga|}}{4}\cdot\delbar(\kappa\cdot \mu)\Big),
\end{multline*}
and
\begin{multline*}
(-1)^{|\al|+|\ga|} {\frac{d^c}{2}} \Big((-1)^{\epsilon+1}(\delbar\lambda)\cdot \nu+(-1)^{\epsilon+|\lambda|+1}\cdot \lambda\cdot (\del \nu)\Big)\\
=(-1)^{|\lambda|+1+\epsilon_0}\cdot \Big((-1)^{|\be|+1}\cdot d^c(\lambda)\Big)\cdot \Big((-1)^{|\de|+1}\cdot d^c(\nu)\Big)+d\Big(\frac{(-1)^{\epsilon+|\al|+|\ga|+1}}{4}\cdot \delbar(\lambda\cdot \nu)\Big).
\end{multline*}
These are the terms on the righthand side of \eqref{EQU:DeRham-inclusion} for the above cyclic defining system up to the $d$-exact terms in the last two equations. We thus get the inclusion \eqref{EQU:DeRham-inclusion} in $H^\bu(A)$.
\end{proof}

The following example shows that the cyclic ABC Massey products are, in general, non-trivial. The underlying algebra for this example comes from the holomorphic analogue of the filiform nilmanifold, which also appeared in \cite[Example 3.14]{SM} for a non-trivial quadruple ABC Massey product. As was pointed out in \cite{SM}, this is a model for a complex nilmanifold which is a holomorphic torus bundle over the Iwasawa bundle.
\begin{example}\label{EXA:holomorphic-filiform}
Let $A=(\Lambda (x,\xb,y,\yb,z,\zb,w,\wb, dz=xy, dw=xz)$, i.e., the only non-zero differential in degree $1$ are $\del z=xy, \delbar \zb=\xb\yb, \del w=xz, \delbar \wb=\xb\zb$. We compute the cyclic ABC Massey product for $\laa x, x\yb, y\yb, x\xb \raa$. Observe that the following is a cyclic ABC defining system:
\[
\begin{tabular}{ccccc}  \cline{3-5}
& \multicolumn{1}{c|}{} & \multicolumn{1}{c|}{$w\zb\, ,\, 0$}  & \multicolumn{1}{c|}{$\,\,\,0\,\,\,$}  & $\,\,\, x\,\,\,$ \\   \cline{3-4}
& \multicolumn{1}{c|}{} & \multicolumn{1}{c|}{$z\zb$}  & $x\xb$ \\ \cline{1-3}
\multicolumn{1}{|c|}{$\, 0\, ,\, 0\,$} & \multicolumn{1}{c|}{$\,\,\, 0\,\,\,$} & $y\yb$ &  \\ \cline{1-2}
\multicolumn{1}{|c|}{$0$} & $x\yb$ & &   \\ \cline{1-1} 
\multicolumn{1}{|c}{$x$} &  &  &
\end{tabular}
\]
Thus, $[(-x\yb w\zb,0) ]\in \laa x, x\yb, y\yb, x\xb \raa$, which is non-zero in $H^{-1}(\mcS_{4,3}(A))$, since $ -x\yb w\zb$ is not an element of $\del(A^{1,2})+\delbar(A^{2,1})$.

We next also show that the zero class in $H^{-1}(\mcS_{4,3}(A))$ is not an element of $\laa x, x\yb, y\yb, x\xb \raa$. Denote by $\kappa, \lambda, \mu,\nu, \eta,\eta',\varrho,\varrho'$ a general cyclic ABC defining system for $\al=x, \be=x\yb,\ga=y\yb, \de=x\xb$ as in \eqref{EQU:cyclic-ABC-defining-system}. Then, solving $\del\delbar\kappa=\al\be$ and $\del\delbar\lambda=\be\ga$, gives 
\begin{align*}
\kappa=&c_{x}\cdot x+c_{y}\cdot y+c_{z}\cdot z+c_{w}\cdot w \\
\lambda=&c_{x\xb}\cdot x\xb+c_{x\yb}\cdot x\yb+c_{x\zb}\cdot x\zb
+c_{x\wb}\cdot x\wb+c_{y\xb}\cdot y\xb+c_{y\yb}\cdot y\yb \\ 
&+c_{y\zb}\cdot y\zb+c_{y\wb}\cdot y\wb+c_{z\xb}\cdot z\xb
+c_{z\yb}\cdot z\yb+c_{w\xb}\cdot w\xb+c_{w\yb}\cdot w\yb 
\end{align*}
for some constants $c_x,\dots, c_{w\yb}$. To solve $\del\eta+\delbar\eta'= \kappa\cdot \ga-\al\cdot \lambda\in A^{2,1}$, note that $\eta'\in A^{2,0}$ and $\delbar(A^{2,0})=\{0\}$, so that $\eta'$ can be any element in $A^{2,0}$, i.e., for some constants $c_{xy}, \dots, c_{zw}$, 
\[
\eta'=c_{xy}\cdot xy+c_{xz}\cdot xz+c_{xw}\cdot xw+c_{yz}\cdot yz+ c_{yw}\cdot yw+c_{zw}\cdot zw.
\]
Since $\eta\in A^{1,1}$, we see that we can only solve for $\eta$ when $c_z=c_w=c_{w\xb}=c_{w\yb}=0$, and in this case we get, for some new constants $c'_{x\xb}, \dots, c'_{y\wb}$,
\begin{align*}
\eta=& (c_x-c_{y\yb})\cdot z\yb-c_{y\xb}\cdot z\xb
-c_{y\zb}\cdot z\zb-c_{y\wb}\cdot z\wb-c_{z\xb}\cdot w\xb-c_{z\yb}\cdot w\yb \\
&+c'_{x\xb}\cdot x\xb+c'_{x\yb}\cdot x\yb+c'_{x\zb}\cdot x\zb+c'_{x\wb}\cdot x\wb
+c'_{y\xb}\cdot y\xb+c'_{y\yb}\cdot y\yb+c'_{y\zb}\cdot y\zb+c'_{y\wb}\cdot y\wb
\end{align*}

Similarly, we solve $\del\delbar\mu= \ga\cdot \de$ and $\del\delbar\nu= \de\cdot \al $ with
\begin{align*}
\mu=& z\zb+\ct_{x\xb}\cdot x\xb+\ct_{x\yb}\cdot x\yb+\ct_{x\zb}\cdot x\zb
+\ct_{x\wb}\cdot x\wb+\ct_{y\xb}\cdot y\xb+\ct_{y\yb}\cdot y\yb \\ 
&+\ct_{y\zb}\cdot y\zb+\ct_{y\wb}\cdot y\wb+\ct_{z\xb}\cdot z\xb
+\ct_{z\yb}\cdot z\yb+\ct_{w\xb}\cdot w\xb+\ct_{w\yb}\cdot w\yb \\
\nu=& \ct_{x}\cdot x+\ct_{y}\cdot y+\ct_{z}\cdot z+\ct_{w}\cdot w
\end{align*}
now with some more constants $\ct_{x\xb},\dots,\ct_w$. Now, in $\del\varrho+\delbar\varrho'= \mu\cdot \al-\ga\cdot \nu$, we get, as before, that $\varrho'\in A^{2,0}$ is given as any linear combination in $A^{2,0}$, i.e., for some constants $\ct_{xy},\dots,\ct_{zw}$,
\[
\varrho'=\ct_{xy}\cdot xy+\ct_{xz}\cdot xz+\ct_{xw}\cdot xw+\ct_{yz}\cdot yz+ \ct_{yw}\cdot yw+\ct_{zw}\cdot zw.
\]
Next, we can solve for $\varrho\in A^{1,1}$ only when $\ct_z=\ct_w=\ct_{w\xb}=\ct_{w\yb}=0$, in which case we get
\begin{align*}
\varrho=& w\zb+(-\ct_x+\ct_{y\yb})\cdot z\yb+\ct_{y\xb}\cdot z\xb
+\ct_{y\zb}\cdot z\zb+\ct_{y\wb}\cdot z\wb+\ct_{z\xb}\cdot w\xb+\ct_{z\yb}\cdot w\yb \\
&+\ct'_{x\xb}\cdot x\xb+\ct'_{x\yb}\cdot x\yb+\ct'_{x\zb}\cdot x\zb+\ct'_{x\wb}\cdot x\wb
+\ct'_{y\xb}\cdot y\xb+\ct'_{y\yb}\cdot y\yb+\ct'_{y\zb}\cdot y\zb+\ct'_{y\wb}\cdot y\wb
\end{align*}
with some constants $\ct'_{x\xb},\dots, \ct'_{y\wb}$.

With this we compute $\mC_S=(\mC\uo_S, \mC\ut_S)\in A^{2,2}\oplus A^{3,1}$ from \eqref{EQU:CS^1} and \eqref{EQU:CS^2}, and we get that the set of all these $\mC_S$ is equal to $(xw\yb\zb,0)+V^{(1)}\oplus V^{(2)}$, where
\begin{align*}
V^{(1)}&=span( xy\xb \yb, xy\xb \zb, xy\xb \wb, xy\yb \zb, xy\yb \wb, xz\xb \yb, xz\xb \zb, xz\xb \wb, xz\yb \zb, xz\yb \wb, xw\xb \yb ) \\
V^{(2)}&=span(xyz\xb, xyw\xb, xzw\xb, xyz\yb, xyw\yb, xzw\yb)
\end{align*}
Note, that the induced classes $[\mC_S]$ cannot be zero in $H^{-1}(\mcS_{4,3}(A))$, since, as noted above, $xw\yb\zb$ is not an element of $\del(A^{1,2})+\delbar(A^{2,1})$.

This shows that $\laa x, x\yb, y\yb, x\xb \raa$ is non-trivial. Moreover, using \eqref{EQU:DeRham-inclusion}, we see that the cyclic Massey product $\la x, x\yb, y\yb, x\xb \ra$ is trivial, since $(\del-\delbar)(xw\yb\zb)=0$ and so it contains $0$.
\end{example}

\bibliographystyle{alpha}
\bibliography{biblio}

\begin{thebibliography}{PTW25}

\bibitem[AT15]{AT}
Daniele Angella and Adriano Tomassini.
\newblock On {B}ott-{C}hern cohomology and formality.
\newblock {\em J. Geom. Phys.}, 93:52--61, 2015.

\bibitem[Mas58]{Mas58}
W.~S. Massey.
\newblock Some higher order cohomology operations.
\newblock In {\em Symposium internacional de topolog\'ia algebraica
  {I}nternational symposium on algebraic topology}, pages 145--154. Universidad
  Nacional Aut\'onoma de M\'exico and UNESCO, M\'exico, 1958.

\bibitem[Mas69]{Mas68}
W.~S. Massey.
\newblock Higher order linking numbers.
\newblock In {\em Conf. on {A}lgebraic {T}opology ({U}niv. of {I}llinois at
  {C}hicago {C}ircle, {C}hicago, {I}ll., 1968)}, pages 174--205. University of
  Illinois at Chicago Circle, Chicago, IL, 1969.

\bibitem[MS24]{SM}
Aleksandar Milivojevi\'c and Jonas Stelzig.
\newblock Bigraded notions of formality and {A}eppli-{B}ott-{C}hern-{M}assey
  products.
\newblock {\em Comm. Anal. Geom.}, 32(10):2901--2933, 2024.

\bibitem[PTW25]{PTW}
Kate Poirier, Thomas Tradler, and Scott~O. Wilson.
\newblock Massey products for homotopy inner products.
\newblock {\em Preprint arXiv:2507.15494}, 2025.

\end{thebibliography}

\end{document}